\documentclass{amsart}

\makeindex

\usepackage[all, cmtip]{xy}
\usepackage{amssymb}
\usepackage{amsmath}
\usepackage{amsthm}
\usepackage{amscd}
\usepackage{amsfonts}
\usepackage{amssymb}
\usepackage{multirow}
\usepackage[usenames,dvipsnames,svgnames,table]{xcolor}
\usepackage{graphicx}
\usepackage{enumerate}

\usepackage[bookmarks=true,bookmarksnumbered=true,
 pdftitle={},
 pdfsubject={},
 pdfauthor={},
 pdfkeywords={TeX; dvipdfmx; hyperref; color;},
 colorlinks=true, linkcolor=BrickRed, citecolor=RoyalBlue]
 {hyperref}

\usepackage{longtable}
\usepackage{pgf,pgfarrows,pgfnodes,pgfautomata,pgfheaps,pgfshade}
\usepackage{tikz}

\usepackage{lmodern}

\makeatletter
\ifcase \@ptsize \relax% 10pt
  \newcommand{\miniscule}{\@setfontsize\miniscule{3}{7}}% \tiny: 5/6
    \newcommand{\stiny}{\@setfontsize\miniscule{5}{7}}% \tiny: 6/7
\or% 11pt
  \newcommand{\miniscule}{\@setfontsize\miniscule{3}{7}}% \tiny: 6/7
   \newcommand{\stiny}{\@setfontsize\miniscule{5}{7}}% \tiny: 6/7
\or% 12pt
  \newcommand{\miniscule}{\@setfontsize\miniscule{3}{7}}% \tiny: 6/7
    \newcommand{\stiny}{\@setfontsize\miniscule{5}{7}}% \tiny: 6/7
\fi
\makeatother

\theoremstyle{definition}

\theoremstyle{remark}

\theoremstyle{corollary}

\theoremstyle{theorem}

\theoremstyle{corollary}

\newtheorem{theorem}{Theorem}[section]
\newtheorem{lemma}[theorem]{Lemma}

\theoremstyle{corollary}

\theoremstyle{definition}

\theoremstyle{remark}
\newtheorem{remark}[theorem]{Remark}

\numberwithin{equation}{section}

\newcommand{\C}{\mathbb{C}}

\newcommand{\Z}{\mathbb{Z}}

\newcommand{\Q}{\mathbb{Q}}
\def\P{\mathbb{P}}

\def\Pic{\operatorname{Pic}}
\def\Proj{\operatorname{Proj}}

\def\Cox{\operatorname{Cox}}

\def\reg{\operatorname{reg}}

\def\Eff{\operatorname{Eff}}

\def\Mov{\operatorname{Mov}}

\def\Bl{\operatorname{Bl}}

\def\deg{\operatorname{deg}}

\newcommand{\suchthat}{\;\ifnum\currentgrouptype=16 \middle\fi|\;}

\input xy
\xyoption{all}

\title[Log Fano structures and Cox rings]
{Log Fano structures and Cox rings of \\blow-ups of products of projective spaces}
\begin{document}

\author{John Lesieutre}
\address{Department of Mathematics, University of Illinois at Chicago, 851 South Morgan St.,
Chicago, IL 60607, USA}
\email{jdl@uic.edu}

\author{Jinhyung Park}
\address{School of Mathematics, Korea Institute for Advanced Study, 85 Hoegiro, Dongdaemun-gu, Seoul 02455, Republic of Korea}
\email{parkjh13@kias.re.kr}


\date{\today}

%\dedicatory{This paper is dedicated to our advisors.}


\begin{abstract}
The aim of this paper is twofold. Firstly, we determine which blow-ups of products of projective spaces at general points are varieties of Fano type, and give boundary divisors making these spaces log Fano pairs. Secondly, we describe generators of the Cox rings of some cases.
\end{abstract}

\maketitle

%\tableofcontents \setcounter{page}{1}

\section{Introduction}

For integers $a,b,c$ such that $a,c \geq 2, b \geq 1$ and if $c=2$ then $a>2$, let
$$
X_{a,b,c}:=\Bl_{b+c}(\P^{c-1})^{a-1}
$$ 
be the blow-up of $(\P^{c-1})^{a-1}$ at $b+c$ general points of $(\P^{c-1})^{a-1}$. Many authors have studied these spaces in a variety of flavors (see e.g., \cite{AM}, \cite{BDP}, \cite{CDBDGP}, \cite{CT}, \cite{DO}, \cite{M1}, \cite{M2}, \cite{M3}, \cite{SV}). Finite generation of the Cox rings of these spaces is closely related to Nagata's construction for Hilbert's 14th problem (see \cite{M1}, \cite{M2}, \cite{CT}). 
The \emph{Cox ring} of a $\Q$-factorial normal projective variety $X$ with a finitely generated Picard group is defined as
$$
\Cox(X) := \bigoplus_{L \in \Pic(X)} H^0(X, L).
$$
Note that $\Cox(X)$ is finitely generated if and only if $X$ is a Mori dream space (\cite[Proposition 2.9]{HK}).
For more details on Cox rings, see \cite{HK}.
Castravet and Tevelev finally proved that $\Cox(X_{a,b,c})$ is finitely generated if and only if
$\frac{1}{a} + \frac{1}{b} + \frac{1}{c} > 1$ (\cite[Theorem 1.3]{CT}).

It is well known that $X_{2,b,3} = \Bl_{b+3} \P^2$ with $1 \leq b \leq 5$ is a del Pezzo surface, and $X_{2,b,4}=\Bl_{b+4} \P^3$ with $1 \leq b \leq 3$ is a weak Fano $3$-fold. In other words, $-K_{X_{2,b,3}}$ is ample and $-K_{X_{2,b,4}}$ is nef and big. Thus it is natural to ask which positivity property is satisfied by the  anticanonical divisor $-K_{X_{a,b,c}}$ of $X_{a,b,c}$ with $\frac{1}{a} + \frac{1}{b} + \frac{1}{c} > 1$. Note that $-K_{X_{a,b,c}}$ is not nef in general.
For example, consider \(X_{2,b,c} = \Bl_{b+c} \P^{c-1}\) with \(c \geq 5\). For the strict transform $C$ of a line between two of the blown-up points, we have \(-K_{X_{2,b,c}}.C = 4-c<0\).

The Cox ring of a $\Q$-factorial variety of Fano type is finitely generated (\cite[Corollary 1.3.2]{BCHM}).
A variety $X$ is said to be \emph{of Fano type} if there is a boundary divisor $\Delta$ on $X$ such that $(X, \Delta)$ is a \emph{log Fano pair}, i.e., it is a klt pair and $-(K_X+\Delta)$ is ample.
For the basics of singularities of pairs, we refer to \cite{KM}.
If $\frac{1}{a} + \frac{1}{b} + \frac{1}{c} \leq 1$, then the Cox ring of $X_{a,b,c}$ is not finitely generated so that $X_{a,b,c}$ is not of Fano type. For the converse, it was asked by Totaro in \cite[Example 8.6]{T} whether $X_{a,b,c}$ with $\frac{1}{a} + \frac{1}{b} + \frac{1}{c} > 1$ is a variety of Fano type. When $a=2$, this question was affirmatively answered in \cite[Theorem 1.3]{AM}.

The following is the first main result of this paper, which answers the above questions.

\begin{theorem}\label{main}
Assume that $\frac{1}{a} + \frac{1}{b} + \frac{1}{c} > 1$. Then we have the following:\\
\indent$(1)$ $-K_{X_{a,b,c}}$ is movable and big.\\
\indent$(2)$ $X_{a,b,c}$ is of Fano type.
\end{theorem}

%Theorem \ref{main} (1) and (2) are shown in Sections \ref{movsec} and \ref{lfsec}, respectively. 
For the movability of $-K_{X_{a,b,c}}$, we prove that the movable cone $\Mov(X_{a,b,c})$ is fixed by the Weyl group action on $\Pic(X_{a,b,c})$. 
It is worth noting that the fact that \(-K_X\) is big and movable is not itself enough to guarantee that \(X\) is of Fano type even if \(X\) is a Mori dream space: an example is provided by~\cite[Example 5.1]{CG}.  
If we run a $-K_X$-minimal model program, then any end product may have bad singularities. Thus one also has to control potential singularities on a $-K_X$-miniaml model as in \cite{CP}.

Instead of using results of \cite{CP}, we use geometric facts about $\Proj \Cox(X_{a,b,c})$ in \cite[Theorem 1.3]{CT} and a characterization of varieties of Fano type via singularities of Cox rings in \cite[Theorem 1.1]{GOST} to prove the stronger conclusion, Theorem \ref{main} (2).
 % which asserts that a $\Q$-factorial variety is of Fano type if and only if its Cox ring only has log terminal singularities.
In Remark \ref{lfs}, we explain a simple way to choose a boundary divisor $\Delta_{a,b,c}$ on $X_{a,b,c}$ with $\frac{1}{a} + \frac{1}{b} + \frac{1}{c} > 1$ such that $(X_{a,b,c}, \Delta_{a,b,c})$ is a log Fano pair.
Our methods and the choice of boundary divisors are different from those of Araujo and Massarenti in \cite{AM}. They use Mukai's result (\cite{M3}) about the nef chamber decomposition of the movable cone and the description of the Mori cone.

Next we investigate generators of the Cox ring of $X_{a,b,c}$ with $\frac{1}{a} + \frac{1}{b} + \frac{1}{c} > 1$.
It was shown by Mukai (\cite{M2}) that $X_{a,b,c}$ is a small $\Q$-factorial modification of $X_{c,b,a}$, so that the Cox rings of $X_{a,b,c}$ and $X_{c,b,a}$ coincide.
Thus we can assume that $a \leq c$. Then $X_{a,b,c}$ is isomorphic to one of the following:\\
\smallskip
\indent$(1)$ $X_{2,b,3} = \Bl_{b+3} \P^2$ with $1 \leq b \leq 5$\\
\indent$(2)$ $X_{2,2,c} = \Bl_{c+2} \P^{c-1}$ with $c \geq 2$\\
\indent$(3)$ $X_{a,1,c}=\Bl_{c+1}(\P^{c-1})^{a-1}$ with $a \geq c \geq 2$ and $(a,c) \neq (2,2)$\\
\indent$(4)$  $X_{2,3,4}=\Bl_7 \P^3$, $X_{2,3,5}=\Bl_8 \P^4$, $X_{3,2,3}=\Bl_5(\P^2)^2$, $X_{3,2,4}=\Bl_6(\P^3)^2$ $X_{3,2,5}=\Bl_7(\P^4)^2$\\
\smallskip
The generators of Cox rings of varieties from (1), (2), (3) are already determined. For the case (1), $X_{2,b,3}$ is a del Pezzo surface, and the Cox rings are generated by distingushied global sections (see \cite[Theorem 3.2]{BP}). For the case (2), \cite[Theorem 1.2]{CT} explicitely describes the generators. For the case (3), we first recall that $X_{a,1,c}=\Bl_{c+1}(\P^{c-1})^{a-1}$ is the GIT quotient of the Grassmannian $Gr(a, a+c)$. Thus $\Cox(X_{a,1,c})$ is isomorphic to the homogeneous coordinate ring of the Pl\"{u}cker embedding of $Gr(a, a+c)$ (see \cite[Remark 3.9]{CT}). 
The only remaining case is (4). In this paper, we determine generators of $\Cox(X_{2,3,4})$ and $\Cox(X_{3,2,3})$ as follows.

\begin{theorem}\label{cox}
The Cox rings of $X_{2,3,4}=\Bl_7 \P^3$ and $X_{3,2,3}=\Bl_5 (\P^2)^2$ are generated by global sections of effective divisors of anticanonical degree $1$.
\end{theorem}

In fact, the Cox rings of varieties from (1), (2), (3) are also generated by global sections of effective divisors of anticanonical degree $1$. We expect that the same is true for the remaining cases $X_{2,3,5}$, $X_{3,2,4}$, and $X_{3,2,5}$. 
For a computational approach, see Remark \ref{othercases}.
We also discuss the relations between the generators of $\Cox(X_{a,b,c})$ in Remark \ref{relation}.

It would be an interesting problem to determine which blown-up spaces of $(\P^n)^k$ at general linear subspaces are Mori dream spaces or even varieties of Fano type. In \cite{DPU}, the authors show that the blow-up of $\P^3$ at six general lines is a weak Fano variety.

\subsection*{Notations}
Throughout the paper, we work over the field $\C$ of complex numbers, and we use the following notations and assumptions.
Let $a,b,c$ be integers such that $a,c \geq 2, b \geq 1$ and if $c=2$ then $a>2$, and let
$X_{a,b,c}:=\Bl_{b+c}(\P^{c-1})^{a-1}$ be the blow-up of $(\P^{c-1})^{a-1}$ at $b+c$ general points of $(\P^{c-1})^{a-1}$.
In the rest of the paper, we always assume that 
$$
\frac{1}{a} + \frac{1}{b} + \frac{1}{c} > 1
$$ 
so that $X_{a,b,c}$ is a Mori dream space by \cite[Theorem 1.3]{CT}.
Note that $\Pic(X_{a,b,c})$ is freely generated by $H_1, \ldots, H_{a-1}$ and $E_1, \ldots, E_{b+c}$ where $H_i$ is the pull-back of the hyperplane from the $i$-th factor of $(\P^{c-1})^{a-1}$ and $E_j$ is an exceptional divisor of the blow-up.
Following Mukai \cite{M2}, we define a symmetric bilinear form on $\Pic(X_{a,b,c})$ as follows: $(H_i, E_j) = 0,~~(H_i, H_j)=(c-1)-\delta_{i,j},~~(E_i,E_j)=-\delta_{i,j}$.
The \emph{anticanonical degree of a divisor} $D \in \Pic(X_{a,b,c})$ is the integer
$\deg(D) := \frac{1}{ac-a-c}(D, -K_{X_{a,b,c}})$.

\subsection*{Organization}
The organization of this paper is as follows. Sections \ref{movsec} and \ref{lfsec} are devoted to proving Theorem \ref{main}. In Section \ref{coxsec}, we show Theorem \ref{cox}.

\subsection*{Acknowledgements}
The authors would like to thank Elisa Postinghel for useful comments.
The authors also wish to thank the referee for helpful suggestions.

\section{Movability of anticanonical divisor $-K_{X_{a,b,c}}$}\label{movsec}

In this section, we prove Theorem \ref{main} (1). First, we consider the Weyl group action on $\Pic(X_{a,b,c})$ following \cite{M2}. The divisor classes of $E_1-E_2, \ldots, E_{b+c}-E_{b+c-1}$, $H_1-E_1 - \cdots - E_c$, $H_1-H_2, \ldots, H_{a-2} - H_{a-1}$ form a system of simple roots of a finite root system with a Dynkin diagram $T_{a,b,c}$ the orthogonal complement of $-K_{X_{a,b,c}}$ in $\Pic(X_{a,b,c})$ where $T_{a,b,c}$ is the T-shaped tree with legs of length $a,b,c$ with $a+b+c-2$ vertices.
Then the Weyl group $W_{a,b,c}$ of $T_{a,b,c}$ naturally acts on $\Pic(X_{a,b,c})$ by fixing $-K_{a,b,c}$.
The following lemma and Theorem \ref{main} (1) might be known to experts, but we include the whole proof for reader's convenience.

\begin{lemma}\label{movfix}
The action of $W_{a,b,c}$ fixes the movable cone $\Mov(X_{a,b,c})$.
\end{lemma}

\begin{proof}
Put $W:=W_{a,b,c}$ and $X:=X_{a,b,c}$. Consider generators $G_1, \ldots, G_r$ of the semigroup $\Eff(X) \cap \Pic(X)$.
It follows from \cite[Proposition 7.2]{BH} that
$$
\text{Mov}(X)= \bigcap_{1 \leq i \leq r} \text{Cone}(G_1, \ldots, \widehat{G_i} , \ldots, G_r).
$$
By \cite[Lemma 2.3]{CT}, $W$ permutes $\{ G_i \mid 1 \leq i \leq r\}$ so that it also permutes $\{ \text{Cone}(G_1, \ldots, \widehat{G_i}, \ldots, G_r) \mid 1 \leq i \leq r \}$. Thus $W$ fixes $\text{Mov}(X)$.
\end{proof}

\begin{remark}
Lemma~\ref{movfix} can also be viewed more geometrically.  The action of \(W_{a,b,c}\) on \(\Mov(X_{a,b,c})\) is induced by taking strict transforms of divisors under the action of the Cremona involution centered at subsets of the blown-up points.  The induced rational map on the blow-up \(X_{a,b,c}\) has indeterminacy of codimension \(2\) \cite[Lemma VI.4.2]{DO}, and so the strict transform of a movable divisor class is movable.
\end{remark}

\begin{proof}[Proof of Theorem \ref{main} (1)]
We use the same notation in Proof of Lemma \ref{movfix}.
Note that $W$ fixes $G_1 + \cdots + G_r$. Since every $W$-fixed element in $\Pic(X)$ is some multiple of $-K_X$, it follows that
$$
G_1 + \cdots + G_r= -\ell K_X ~~\text{ for some integer $\ell >0$}.
$$
Thus $-K_X$ is big.
Using Lemma \ref{movfix}, one can similarly show the movability of $-K_X$: the sum of the generators of \(\text{Mov}(X) \cap \Pic(X)\) is fixed by the action of $W$, and thus must be a positive multiple of the anticanonical class, which is consequently movable.
%For the movability of $-K_X$, consider generatros of $M_1, \ldots, M_k$ of the semigroup $\text{Mov}(X) \cap \Pic(X)$. By Lemma \ref{movfix}, $W$ fixes $M_1 + \cdots + M_k$. Thus we obtain
%$$M_1 + \cdots + M_k = -\ell' K_X$$
%for some integer $\ell' >0$, and hence, $-K_X$ is movable.
\end{proof}

\section{Log Fano structure of $X_{a,b,c}$}\label{lfsec}

In this section, we first prove Theorem \ref{main} (2), and then explain a way to choose a boundary divisor $\Delta_{a,b,c}$ on $X_{a,b,c}$ such that $(X_{a,b,c}, \Delta_{a,b,c})$ is a log Fano pair in Remark \ref{lfs}.

\begin{proof}[Proof of Theorem \ref{main} (2)]
Put $X:=X_{a,b,c}$. Since $\Cox(X)$ has the natural $\Z$-grading by the anticanonical degree, we can consider the projective variety $Z=\Proj \Cox(X)$.
Then $\Pic(Z)$ is generated by $\mathcal{O}_Z(1)$, and $Z$ is a normal projective Gorenstein variety with rational singularities such that $-K_Z$ is ample (\cite[Theorem 1.3]{CT}). Thus $Z$ is a Fano variety with canonical singularities so that $\Cox(Z)$ has log terminal singularities by \cite[Theorem 1.1]{GOST}. Since $\Cox(X)$ and $\Cox(Z)$ are unique factorization domains (\cite[Corollary 1.2]{EKW}), they are integrally closed. \cite[Claim 3.7]{CT} says that $\Cox(X)$ is integral over the homogeneous coordinate ring $R$ of some embedding $Z \subset \P^N$ whose hyperplane section is $\mathcal{O}_Z(1)$. Note that $\Cox(Z)$ is a section ring of $\mathcal{O}_Z(1)$ so that $\Cox(Z)$ is also integral over $R$. Theorefore, $\Cox(X)=\Cox(Z)$. By applying \cite[Theorem 1.1]{GOST} to $\Cox(X)$, we conclude that $X$ is of Fano type.
\end{proof}

\begin{remark}\label{lfs}
Here we explain a standard way to choose a boundary divisor $\Delta$ on a $\Q$-factorial variety $X$ of Fano type under the assumption that $-K_X$ is big and movable. 
Since $X$ is a Mori dream space by \cite[Corollary 1.3.2]{BCHM}, we can run the $-K_X$-minimal model program and obtain a model $f \colon X \dashrightarrow X'$ such that $X'$ is also a Mori dream space and $-K_{X'}$ is semiample and big (see \cite[Proposition 1.11]{HK}). By Bertini's theorem, there is an irreducible and reduced divisor $H' \in |-mK_{X'}|$ for a sufficiently large integer $m>0$. Then $(X', \frac{1}{m}H')$ is a klt Calabi-Yau pair. A log pair $(X, \Delta)$ is called a \emph{klt Calabi-Yau pair} if it is a klt pair and $K_X+\Delta \sim_{\Q} 0$.
Now $H:=f_{*}^{-1}H'$ is also an irreducible and reduced divisor, and $\frac{1}{m}H \sim -K_X$. Then $(X, \frac{1}{m}H)$ is also a klt Calabi-Yau pair. 
Note that there is a canonical way to identify $|-mK_{X'}|$ with $|-mK_X|$ for every integer $m >0$. Thus we can simply choose a general irreducible and reduced element $H \in |-mK_X|$ for a sufficiently large integer $m>0$.
Next, take an ample $\Q$-divisor $A$ on $X$ such that $-K_X - A$ is linearly equivalent to an effective $\Q$-divisor $B$. For a sufficiently small rational number $\epsilon >0$, let $\Delta:= (1-\epsilon)\frac{1}{m}H + \epsilon B$. Then $(X, \Delta)$ is clearly a klt pair and $-(K_X+\Delta) \sim_{\Q} \epsilon A$ is ample, i.e., $(X, \Delta)$ is a log Fano pair.
\end{remark}

\section{Generators of Cox ring of $X_{a,b,c}$}\label{coxsec}

In this section, we determine generators of $\Cox(X_{2,3,4})$ and $\Cox(X_{3,2,3})$, and explain a computational approach to compute generators of $\Cox(X_{a,b,c})$ in the remaining cases.
%Recall that $\Pic(X_{a,b,c})$ has a canonical choice of generators; $H_1, \ldots, H_{a-1}$ and $E_1, \ldots, E_{b+c}$. Then we define the \emph{Cox ring of $X_{a,b,c}$} as
%$$\Cox(X_{a,b,c}) := \bigoplus_{(i_1, \ldots, i_{a-1}, j_1, \ldots, j_{b+c}) \in \Z^{a+b+c-1}} H^0(X_{a,b,c}, i_1H_1 + \cdots + i_{a-1}H_{a-1} + j_1E_1 + \cdots + j_{b+c}E_{b+c})$$
An effective divisor $E$ on $X_{a,b,c}$ is called a \emph{$(-1)$-divisor} in the sense of Mukai if $(E,E)=-1$ and $\deg(E) = \frac{1}{ac-a-c}(E, -K_{X_{a,b,c}})=1$.
According to Mukai~\cite{M2}, a divisor $E$ is a \((-1)\)-divisor if and only if there exists a small \(\Q\)-factorial modification $X'$ of \(X_{a,b,c}\) on which the strict transform of \(E\) is the exceptional divisor of some contraction. We can assume that $X'$ is also a blow-up of $(\P^{c-1})^{a-1}$ in $b+c$ general points (see e.g., \cite[Lemma 2.2]{CT}).

\begin{lemma}\label{novan}
Assume that $X = X_{2,3,4}$ or $X_{3,2,3}$.
If $D$ is a divisor on $X$ such that $(D.E) \geq 0$ for every $(-1)$-divisor $E$ and $(D.E_0)=0$ for some $(-1)$-divisor $E_0$, then there is a global section of $H^0(X, D)$ which vanishes nowhere on $E_0$.
\end{lemma}

\begin{proof}
We proceed with the induction on the number of blown-up points. First, we consider the case $X = X_{2,3,4} = \Bl_7 \P^3$, which can be contracted to \(\Bl_k(\P^3)\). 
For $1 \leq k \leq 7$, let $X':=\Bl_k \P^3$. The assertion for the base case $k=1$ is clear.
Consider the divisor $L:=-\frac{1}{2}K_{X'} = 2H_1-\sum_{i=1}^k {E_i}$ which satisfies $(L, E)=1$ for all $(-1)$-divisors $E$. Note that $H^0(X', L) \neq 0$.
Since $(D,E_0)=0$, we see that $D|_{E_0}=0$. Thus it is sufficient to find a global section of $H^0(X'', D')$ which vanishes nowhere on $E_0'$ where $X''$ is a small $\Q$-factorial modification of $X'$ and $D', E_0'$ are strict transforms of $D, E_0$.

After passing to a small $\Q$-factorial modification, we may assume that the blow-down $f \colon X' \to Y$ contracts $E_0$, and that on this model, the strict transform of \(E_0\) is isomorphic to \(\P^{2}\). Recall that $E_1, \ldots, E_k$ are exceptional divisors of the blow-up $X' \to \P^3$. For simplicity, we put $E_1=E_0$.
If \(D \sim d H_1 - \sum_{i=1}^k m_i E_i\), then by the definition of the Mukai pairing, we have \(m_1=(D,E_0)\). In particular, since \((D,E_0) = 0\), the divisor \(D\) is the pullback of a divisor on \(Y\).

By assumption, there is a divisor $D'$ on $Y$ with $f^*D' = D$. Then $(D'.E') \geq 0$ for every $(-1)$-divisor $E'$ on $Y$. If there is a $(-1)$-divisor $E'_0$ on $Y$ such that $(D', E'_0)=0$, then $D'$ is effective by the induction. Since we blow up a general point, the conclusion immediately follows.
Now we consider the case that $(D'.E')>0$ for every $(-1)$-divisor $E'$ on $Y$.
Take $m$ to be the minimum of numbers $(D', E')$ for all $(-1)$-divisors $E'$ on $Y$. Let $E'_0$ satisfy $(D', E'_0)=m$. By abuse of notation, we let $L=-\frac{1}{2}K_{Y}$.
By the induction, $D'':=D'-mL$ has a nontrivial global section, and so there is a nonzero section  $D'=D'' +mL$. By taking pull-back of this global section, we get a global section of $D=f^*D'$ which does not vanish at any point of $E_0$.
The same argument works for the case $X = X_{3,2,3}$ using the divisor $L:=-\frac{1}{3}K_X=H_1 + H_2 - \sum_{i=1}^{5}{E_i}$. 
\end{proof}

\begin{proof}[Proof of Theorem \ref{cox}]
We closely follow \cite[Proof of Theorem 3.2]{BP}.
Let $X = X_{2,3,4}$ or $X_{3,2,3}$, and let $Y = X_{2,2,4}$ or $X_{3,1,3}$ with a blow-up $f \colon X \to Y$.
As we saw in the introduction, the assertion holds for $Y$.
It suffices to show that the following claim: For an effective divisor $D$ on $X$ with $\deg(D) \geq 2$, the vector space $H^0(X, D)$ is spanned by global sections of anticanonical degree one divisors. 
Now we proceed an induction on $\deg (D)$.

Case 1. $(D, E) < 0$ for some $(-1)$-divisor $E$. Then $H^0(X, D-E)=H^0(X, D)$. Since $\deg (D-E) = \deg (D)-1$, the claim for this case follows from the induction.

Case 2. $(D, E)=0$ for some $(-1)$-divisor $E$. As in the proof of Lemma~\ref{novan}, by considering a small $\Q$-factorial modification, we may assume that the blow-down $f \colon X \to Y$ contracts $E$. Then $D=f^*D'$ for some effective divisor $D'$ on $Y$. Then $H^0(X,D) = H^0(Y,D')$ and the claim holds in this case.

Case 3. $(D, E)>0$ for every $(-1)$-divisor $E$. Let $m$ be the minimum of numbers $(D, E)$ for all $(-1)$-divisors $E$, and let $E_0$ be a $(-1)$-divisor such that $(D, E_0)=m$. Put $D' := D-E_0$.
By considering a small $\Q$-factorial modification, we may assume that the blow-down $f \colon X \to Y$ contracts $E_0$. In particular, $E_0 \simeq \P^{n}$ where $n:=\dim X-1$.
Consider an exact sequence
$$
0 \to H^0(X, D') \to H^0(X, D) \to H^0(E_0, D|_{E_0}) \to \cdots.
$$
By the induction, $H^0(X, D')$ is spanned by global sections of anticanonical degree one divisors. Thus we only have to show that there exist global sections of $H^0(X, D)$ which are generated by global sections of anticanonical degree one divisors and generate $H^0(E_0, D|_{E_0})$.
Since $(D,E_0)=m$, we can naturally identify $H^0(E_0, D|_{E_0})$ with $H^0(E_0, mL|_{E_0}) = H^0(\P^{n}, \mathcal{O}_{\P^n}(m))$ where $L = 2H_1-\sum_{i=1}^{7}{E_i}$ or $H_1 + H_2 - \sum_{i=1}^{5}{E_i}$ as in Proof of Lemma \ref{novan}.
Since $(D-mL, E) \geq 0$ for every $(-1)$-divisor $E$ and $(D-mL, E_0)=0$, it follows from Lemma \ref{novan} that there is a global section $s \in H^0(X, D-mL)$ such that $s$ nowhere vanishes on $E_0$. We now consider the following commutative diagram
\[
\xymatrix{
	H^0(X, mL) \ar[r]^{\cdot s} \ar[d] & H^0(X, D) \ar[d] \\
	H^0(E_0, mL|_{E_0}) \ar[r] &  H^0(E_0, D|_{E_0}).
	}
\]
The bottom map of the above diagram is an isomorphism because $s$ nowhere vanishes on $E_0$.
It remains to show that there exist global sections of $H^0(X, mL)$ generated by global sections of anticanonical degree one divisors such that their restrictions to $E_0$ generate $H^0(E_0, mL|_{E_0})$. Note that this statement is trivial when $m=1$ since $L$ is already a degree one divisor. In particular, one can say that $H^0(E_0, L|_{E_0})=H^0(E_0, D|_{E_0})=H^0(\P^n, \mathcal{O}_{\P^n}(1))$ is generated by global sections of anticanonical degree one divisors.
Then the same is true for the space $H^0(E_0, mD|_{E_0}) = H^0(\P^n, \mathcal{O}_{\P^n}(m))=\text{Sym}^m H^0(\P^n, \mathcal{O}_{\P^n}(1))$. Thus we complete the proof.
\end{proof}

\begin{remark}\label{othercases}
We do not determine generators of the Cox ring in the three cases $X=X_{2,3,5}$, $X_{3,2,4}$, and $X_{3,2,5}$. However, we expect that in all these cases the Cox ring is generated by divisors of anticanonical degree \(1\). This could in principle be verified by a finite computation, which we outline below; however, this seems to be impractical without additional optimizations.

The first step is to compute the set of divisors on \(X\) with anticanonical degree \(1\).  Every $(-1)$-divisor is an effective divisor of anticanonical degree $1$. According to Mukai~\cite{M2}, these divisors are precisely the orbit of one of the exceptional divisors \(E_0\) under the Weyl group. On $X_{2,3,5} = \Bl_8 \P^4$, there are $2160$ such divisors, fitting into $15$ distinct classes up to the action of $S_8$.  Since the $(-1)$-divisors generate \(\Eff(X_{a,b,c})\), any effective divisor of anticanonical degree \(1\) is of the form \(\sum_i a_i E_i\), where \(E_i\) ranges over the \((-1)\)-divisors and the \(a_i\) are non-negative rational numbers.  The anticanonical degree of \(\sum_i a_i E_i\)  is \(\sum_i a_i\), and so every divisor of anticanonical degree \(1\) has coefficients in the compact set \(0 \leq a_i \leq 1\).  In particular, there are only finitely many integral effective divisors \(F_j\) of anticanonical degree \(1\).  For each of these, compute \(H^0(X,F_j)\).

Given any divisor $D$, one can compute the number of distinct representations of $D$ as a linear combination of anticanonical degree one divisors \(F_j\), and so compute the dimension of the image \(\bigoplus_j H^0(X,\mathcal O_X(F_j)) \to H^0(X,D)\).  For a given divisor \(D\), one can also compute \(H^0(X,D)\) directly, and thereby check whether this map is surjective.

Now we argue that it is necessary to check this computation only for effective divisors with anticanonical degree less than some bound.
Recall that $Z_{a,b,c}=\Proj \Cox(X_{a,b,c})$ is a Fano variety with canonical singularities by \cite[Theorem 1.3]{CT}. By Kodaira vanishing theorem, we get the Castelnuovo-Mumford regularity as $\reg \mathcal{O}_{Z_{a,b,c}}=\reg \Cox(X_{a,b,c})=b(a-1)(c-1)-1$. Then the maximal anticanonical degree $m_{a,b,c}$ of minimal generators of $\Cox(X_{a,b,c})$ is at most $b(a-1)(c-1)$. By further considering Green's duality and $K_{p,1}$-theorem (\cite[Theorems 2.c.6 and 3.c.1]{G}), we see that $m_{a,b,c} \leq b(a-1)(c-1)-2$. 
For $X_{2,3,5} = \Bl_8 \P^4$, we have $m_{2,3,5} \leq 10$. If we find that all the sections arise as combinations of anticanonical degree \(1\) divisors for all \(D\) with anticanonical degree $\leq  b(a-1)(c-1)-2$, then our expectation would immediately follow.
\end{remark}

\begin{remark}
Elisa Postinghel asks the geometric description of the generators of Cox rings of $X_{a,b,c}$. On $X_{2,3,4}=\Bl_7 \P^3$, the $(-1)$-divisor $2H_1 - 2E_1 - E_2-E_3-E_4-E_5-E_6$ can be interpreted as the strict transform of the pointed cone with vertex the first blown-up point over the twisted cubic curve through the first six blown-up points. But we do not know such a geometric description in general.
\end{remark}

\begin{remark}\label{relation}
Recall that $\reg \Cox(X_{a,b,c})=b(a-1)(c-1)-1$. By considering again Green's duality and $K_{p,1}$-theorem, we see that the maximal degree of minimal generators of the ideal $I_{a,b,c}$ of relations of generators of $\Cox(X_{a,b,c})$ is $\leq b(a-1)(c-1)-2$. But we expect that the ideal $I_{a,b,c}$ is generated by quadrics as in the del Pezzo surface cases (cf. \cite{BP}, \cite[Theorem 1.1]{TVAV}).
\end{remark}

$ $
%%%%%%%%%%%%%%%%%%%%%%%%%%%%%%%%%%%%%%%%%%%%%%%%%%%%%%%%%%%%%%%%%%%%%%%%%%%%%%%%%%%%%%%%%%%%%%%%%%%%%%%%
%BIBLIOGRAFIA

\end{document}